\documentclass{article}

\usepackage{amsmath,amsthm,bm,graphicx,xcolor,float}
\usepackage[hyphens]{url}

\title{How strong can the Parrondo effect be?  II} 
\author{S. N. Ethier\thanks{Department of Mathematics, University of Utah, 155 S. 1400 E., Salt Lake City, UT 84112.}\ \ and Jiyeon Lee\thanks{Department of Statistics, Yeungnam University, 280 Daehak-Ro, Gyeongsan, Gyeongbuk 38541, South Korea.}} 
\date{}
                                   
\newtheorem{theorem}{Theorem}
\newtheorem{lemma}[theorem]{Lemma}
\newtheorem{corollary}[theorem]{Corollary}
\theoremstyle{definition}

\newtheorem{example}{Example}

\DeclareMathOperator*{\argmax}{{arg\,max}} 

\newenvironment{newreferences}
               {\section*{References}
                \raggedright
                \begin{list}{}{\setlength{\itemsep}{0pt}
                               \setlength{\parsep}{0pt}
                               \setlength{\labelwidth}{0pt}
                               \setlength{\leftmargin}{12pt}
                               \setlength{\labelsep}{0pt}}
                \setlength{\itemindent}{-12pt}
               }{\end{list}}

\allowdisplaybreaks    

\begin{document}
\maketitle

\begin{abstract}
Parrondo's coin-tossing games comprise two games, $A$ and $B$.  The result of game $A$ is determined by the toss of a fair coin.  The result of game $B$ is determined by the toss of a $p_0$-coin if capital is a multiple of $r$, and by the toss of a $p_1$-coin otherwise.  In either game, the player wins one unit with heads and loses one unit with tails.  Game $B$ is fair if $(1-p_0)(1-p_1)^{r-1}=p_0\,p_1^{r-1}$.  In a previous paper we showed that, if the parameters of game $B$, namely $r$, $p_0$, and $p_1$, are allowed to be arbitrary, subject to the fairness constraint, and if the two (fair) games $A$ and $B$ are played in an arbitrary periodic sequence, then the rate of profit can not only be positive (the so-called Parrondo effect), but also be arbitrarily close to 1 (i.e., 100\%).  Here we prove the same conclusion for a random sequence of the two games instead of a periodic one, that is, at each turn game $A$ is played with probability $\gamma$ and game $B$ is played otherwise, where $\gamma\in(0,1)$ is arbitrary.\medskip

\medskip

\end{abstract}

\section{Introduction}\label{sec-intro}

The flashing Brownian ratchet of Ajdari and Prost (1992) is a stochastic model in statistical physics that is also of interest to biologists in connection with so-called molecular motors.  In 1996 J. M. R. Parrondo proposed a toy model of the flashing Brownian ratchet involving two coin-tossing games.  Both of the games, $A$ and $B$, are individually fair or losing, whereas the random mixture (toss a fair coin to determine whether game $A$ or game $B$ is played) is winning, as are periodic sequences of the games, such as $AABB\,AABB\,AABB\,\cdots$.  
 
Harmer and Abbott (1999) described the games explicitly.  For simplicity, we omit the bias parameter, so that both games are fair.  Let us define a $p$-coin to be a coin with probability $p$ of heads.  In Parrondo's original games, game $A$ uses a fair coin, while game $B$ uses two biased coins, a $p_0$-coin if capital is a multiple of 3 and a $p_1$-coin otherwise, where
\begin{equation}\label{p0,p1}
p_0=\frac{1}{10}\quad\text{and}\quad p_1=\frac34.
\end{equation}
The player wins one unit with heads and loses one unit with tails.  Both games are fair, but the random mixture, denoted by $\frac12 A+\frac12 B$, has long-term cumulative profit per game played (hereafter, rate of profit) 
\begin{equation}\label{(A+B)/2}
\mu\big({\textstyle\frac12}A+{\textstyle\frac12}B\big)=\frac{18}{709}\approx0.0253879,
\end{equation}
and the pattern $AABB$, repeated ad infinitum, has rate of profit
\begin{equation}\label{AABB}
\mu(AABB)=\frac{4}{163}\approx0.0245399.
\end{equation}
Dinis (2008) found that the pattern $ABABB$ (or any cyclic permutation of it) has the highest rate of profit, namely
\begin{equation}\label{ABABB}
\mu(ABABB)=\frac{3613392}{47747645}\approx0.0756769.
\end{equation}

How large can these rates of profit be if we vary the parameters of the games, subject to a fairness constraint?  

Game $A$ is always the same fair-coin-tossing game.  With $r\ge3$ an integer, game $B$ is a mod~$r$ capital-dependent game that uses two biased coins, a $p_0$-coin ($p_0<\frac12$) if capital is a multiple of $r$, and a $p_1$-coin ($p_1>\frac12$) otherwise.  The probabilities $p_0$ and $p_1$ must be such that game $B$ is fair,  requiring the constraint
$$
(1-p_0)(1-p_1)^{r-1}=p_0\, p_1^{r-1},
$$
or equivalently,
\begin{equation}\label{rho-param}
p_0=\frac{\rho^{r-1}}{1+\rho^{r-1}}\quad\text{and}\quad p_1=\frac{1}{1+\rho}
\end{equation}
for some $\rho\in(0,1)$.  The special case of $r=3$ and $\rho=\frac13$ gives \eqref{p0,p1}.  

The games are played randomly or periodically.  Specifically, we consider the random mixture $\gamma A+(1-\gamma)B$ (game $A$ is played with probability $\gamma$ and game $B$ is played otherwise) as well as the pattern $\Gamma(A,B)$, repeated ad infinitum.  We denote the rate of profit by 
$$
\mu(r,\rho,\gamma A+(1-\gamma)B)\quad\text{or}\quad\mu(r,\rho,\Gamma(A,B)), 
$$
so that the rates of profit in \eqref{(A+B)/2}--\eqref{ABABB} in this notation become $\mu(3,\frac13,\frac12 A+\frac12 B)$, $\mu(3,\frac13,AABB)$, and $\mu(3,\frac13,ABABB)$.  

How large can $\mu(r,\rho,\gamma A+(1-\gamma)B)$ and $\mu(r,\rho,\Gamma(A,B))$ be?  The answer, at least in the second case, is that it can be arbitrarily close to $1$ (i.e., $100$\%):

\begin{theorem}[Ethier and Lee (2019)]\label{sup=1,periodic}
\begin{equation*}
\sup_{r\ge3,\;\rho\in(0,1),\;\Gamma(A,B)\text{ arbitrary}}\mu(r,\rho,\Gamma(A,B))=1.
\end{equation*}
\end{theorem}

In the first case the question was left open, and it is the aim of this paper to resolve that issue.  It turns out that the conclusion is the same:

\begin{theorem}\label{sup=1,random}
\begin{equation*}
\sup_{r\ge3,\;\rho\in(0,1),\;\gamma\in(0,1)}\mu(r,\rho,\gamma A+(1-\gamma)B)=1.
\end{equation*}
\end{theorem}
This will be seen to be a consequence of Corollary~\ref{rate-random-asymp} below.

We can compute $\mu(r, \rho, 􏰁\gamma A+(1-\gamma)B)$ and $\mu(r,\rho,\Gamma(A,B))$ for $r\ge3$, $\rho\in(0,1)$, $\gamma\in(0,1)$, and patterns $\Gamma(A,B)$. Indeed, the method of Ethier and Lee (2009) applies if $r$ is odd, and generalizations of it apply if $r$ is even; see Section 2 for details in the random mixture case and Ethier and Lee (2019) in the periodic pattern case. For example,
\begin{equation}\label{mu(3,rho,(A+B)/2)}
\mu(3, \rho, {\textstyle􏰁\frac12} A+{\textstyle\frac12} B)=\frac{9 (1 - \rho)^3 (1 + \rho)}{2 (35 + 70 \rho + 78 \rho^2 + 70 \rho^3 + 35 \rho^4)}
\end{equation}
and
\begin{equation}\label{mu(3,rho,AABB)}
\mu(3, \rho, AABB)=\frac{3 (1 - \rho)^3 (1 + \rho)}{8 (3 + 6 \rho + 7 \rho^2 + 6 \rho^3 + 3 \rho^4)}.
\end{equation}
These and other examples suggest that typically $\mu(r, \rho,􏰁\gamma A+(1-\gamma)B)$ and $\mu(r,\rho,\Gamma(A,B))$ are decreasing in $\rho$ (for fixed $r$, $\gamma$, and $\Gamma(A,B)$), hence maximized at $\rho = 0$. We excluded the case $\rho = 0$ in \eqref{rho-param}, but now we want to include it. We find from \eqref{mu(3,rho,(A+B)/2)} and \eqref{mu(3,rho,AABB)} that
\begin{equation*}
\mu(3, 0, {\textstyle􏰁\frac12} A+{\textstyle\frac12} B)=\frac{9}{70}\approx0.128571\quad\text{and}\quad\mu(3, 0,AABB)=\frac18=0.125,
\end{equation*} 
which are substantial increases over $\mu(3,\frac13,\frac12 A+\frac12 B)$ and $\mu(3,\frac13,AABB)$ (see \eqref{(A+B)/2} and \eqref{AABB}).  We can do slightly better by choosing $\gamma$ optimally:
\begin{equation*}
\mu(3, 0, \gamma A+(1-\gamma)B)=\frac{3 \gamma(1 - \gamma)(2 - \gamma)}{(2+\gamma)(4-\gamma)},
\end{equation*} 
so $\max_\gamma\mu(3, 0, \gamma A+(1-\gamma)B)\approx0.133369$, achieved at $\gamma\approx0.407641$.  Similarly, we can do considerably better by choosing the optimal pattern $ABABB$:
\begin{equation}\label{mu(3,0,ABABB)}
\mu(3,0,ABABB)=\frac{9}{25}=0.36.  
\end{equation}
Thus, we take $\rho = 0$ in what follows.

Theorem~\ref{sup=1,periodic} was shown to follow from the next theorem.

\begin{theorem}[Ethier and Lee (2019)]\label{rate-periodic}
Let $r\ge3$ be an odd integer and $s$ be a positive integer.  Then
\begin{equation*}\label{formula1}
\mu(r,0,(AB)^sB^{r-2})=\frac{r}{2s+r-2}\;\frac{2^s-1}{2^s+1},
\end{equation*}
regardless of initial capital.

Let $r\ge4$ be an even integer and $s$ be a positive integer.  Then
\begin{equation*}\label{formula2}
\mu(r,0,(AB)^sB^{r-2})=\begin{cases}\cfrac{r}{2s+r-2}\;\displaystyle{\sum_{k=0}^s\bigg\lceil\frac{2k}{r}\bigg\rceil\binom{s}{k}\frac{1}{2^s}}&\text{if initial capital is even},\\ \noalign{\medskip}
0&\text{if initial capital is odd}.\end{cases}
\end{equation*}
\end{theorem}

The special case $(r,s)=(3,2)$ of this theorem is consistent with \eqref{mu(3,0,ABABB)}.

Theorem~\ref{sup=1,random} will be seen to follow from the next two results, the proofs of which are deferred to Section \ref{sec-rate}.

\begin{theorem}\label{rate-random}
Let $r\ge3$ be an integer and $0<\gamma<1$. Then
\begin{equation*}
\mu(r,0,\gamma A+(1-\gamma)B)= \frac{r \gamma(1-\gamma)(2-\gamma)[(2-\gamma)^{r-2}-\gamma^{r-2}]}{2[(2-\gamma)^{r}-\gamma^{r}]+ r \gamma (2-\gamma) [(2-\gamma)^{r-2}-\gamma^{r-2}]},
\end{equation*}
regardless of initial capital.
\end{theorem}

\begin{corollary}\label{rate-random-asymp}
For each integer $r\ge3$, define $\gamma_r:=2/\sqrt{r}$. Then
\begin{equation*}
1-\mu(r,0,\gamma_r A+(1-\gamma_r)B)\sim2\gamma_r\text{ as }r\to\infty,
\end{equation*}
regardless of initial capital.
\end{corollary}

Table~\ref{rates-random} illustrates these results.

\begin{table}[H]
\caption{\label{rates-random}The rate of profit $\mu(r,0,\gamma A+(1-\gamma)B)$.}
\catcode`@=\active \def@{\hphantom{0}}
\tabcolsep=.2cm
\small
\begin{center}
\begin{tabular}{rcccccc}
\hline
\noalign{\smallskip}
$r$ &@& @$\argmax_\gamma\, \mu$@@ & $1-\max_\gamma \mu$ &@@&  $\gamma_r:=2/\sqrt{r}$ & $1-\mu$ at $\gamma=\gamma_r$   \\
\noalign{\smallskip}
\hline
\noalign{\smallskip}
10 && 0.366017@@ & 0.665064@@ && 0.632456@@ & 0.743544@@ \\
100 && 0.165296@@ & 0.316931@@ && 0.200000@@ & 0.322034@@ \\
1000 && 0.0594276@ & 0.117089@@ && 0.0632456@ & 0.117307@@ \\
10000 && 0.0196059@ & 0.0390196@ && 0.0200000@ & 0.0390273@ \\
100000 && 0.00628474 & 0.0125497@ && 0.00632456 & 0.0125500@ \\
1000000 && 0.00199601 & 0.00399002 && 0.00200000 & 0.00399003 \\
\noalign{\smallskip}
\hline
\end{tabular}
\end{center}
\end{table}

For the purpose of comparison, let us state a corollary to Theorem~\ref{rate-periodic} that is analogous to Corollary~\ref{rate-random-asymp}.

\begin{corollary}\label{rate-periodic-asymp}
For each integer $r\ge3$, define $s_r:=\lfloor\log_2 r\rfloor-1$. Then
\begin{equation*}
1-\mu(r,0,(AB)^{s_r}B^{r-2})\sim\frac{2s_r}{r}\text{ as }r\to\infty,
\end{equation*}
assuming initial capital is even if $r$ is even,
\end{corollary}

Table~\ref{rates-periodic} illustrates Theorem~\ref{rate-periodic} and Corollary~\ref{rate-periodic-asymp}.

\begin{table}[H]
\caption{\label{rates-periodic}The rate of profit $\mu(r,0,(AB)^s B^{r-2})$, assuming initial capital is even.}
\catcode`@=\active \def@{\hphantom{0}}
\tabcolsep=.2cm
\small
\begin{center}
\begin{tabular}{rcccc}
\hline
\noalign{\smallskip}
$r$ &@& @$\argmax_s \mu$@@ & $1-\max_s \mu$ &  $s_r:=\lfloor\log_2 r\rfloor-1$   \\
\noalign{\smallskip}
\hline
\noalign{\smallskip}
10 && $2,3$ & 0.375000@@@@ & 2  \\
100 && 5 & 0.103009@@@@ & 5  \\
1000 && 8 & 0.0176590@@@ & 8  \\
10000 && 12 & 0.00243878@@ & 12  \\
100000 && 15 & 0.000310431@ & 15  \\
1000000 && 18 & 0.0000378134 & 18  \\
\noalign{\smallskip}
\hline
\end{tabular}
\end{center}
\end{table}

Ethier and Lee (2019) remarked that the rates of profit of periodic sequences tend to be larger than those of random sequences.  Corollaries~\ref{rate-random-asymp} and \ref{rate-periodic-asymp} yield a precise formulation of this conclusion.

\section{SLLN for random sequences of games}\label{sec-SLLN}

Ethier and Lee (2009) proved a strong law of large numbers (SLLN) and a central limit theorem for random sequences of Parrondo games.  It is only the SLLN that is needed here. 

\begin{theorem}[Ethier and Lee (2009)]\label{SLLN}
Let $\bm P$ be the transition matrix for a Markov chain in a finite state space $\Sigma$. Assume that $\bm P$ is irreducible and aperiodic, and let the row vector $\bm\pi$ be the unique stationary distribution of $\bm P$.  Given a real-valued function $w$ on $\Sigma\times\Sigma$,  define the payoff matrix $\bm W:=(w(i,j))_{i,j\in\Sigma}$, and put
\begin{equation*}\label{mean-formula}
\mu:=\bm\pi\dot{\bm P}\bm1,
\end{equation*}
where $\dot{\bm P}:=\bm P\circ\bm W$ (the Hadamard, or entrywise, product), and $\bm1$ denotes a column vector of $1$s with entries indexed by $\Sigma$.  Let $\{X_n\}_{n\ge0}$ be a Markov chain in $\Sigma$ with transition matrix $\bm P$, and let the initial distribution be arbitrary.  For each $n\ge1$, define $\xi_n:=w(X_{n-1},X_n)$ and $S_n:=\xi_1+\cdots+\xi_n$.  Then $\lim_{n\to\infty}n^{-1}S_n=\mu$ a.s.
\end{theorem}

We wish to apply Theorem~\ref{SLLN} with 
\begin{equation*}\label{Sigma}
\Sigma=\{0,1,\ldots,r-1\}
\end{equation*}
($r$ is the modulo number in game $B$), $\bm P:=\gamma \bm P_A+(1-\gamma)\bm P_B$, where
the $r\times r$ transition matrices $\bm P_A$ and $\bm P_B$ are given by
\begin{equation*}\label{PA}
{\bm P}_A=\begin{pmatrix}0&\frac12&0&0&\cdots&0&0&0&\frac12\\
\frac12&0&\frac12&0&\cdots&0&0&0&0\\
0&\frac12&0&\frac12&\cdots&0&0&0&0\\
\vdots&\vdots&\vdots&\vdots& &\vdots&\vdots&\vdots&\vdots\\
0&0&0&0&\cdots&\frac12&0&\frac12&0\\
0&0&0&0&\cdots&0&\frac12&0&\frac12\\
\frac12&0&0&0&\cdots&0&0&\frac12&0\end{pmatrix}
\end{equation*}
and
\begin{equation*}\label{PB}
{\bm P}_B=\begin{pmatrix}0&p_0&0&0&\cdots&0&0&0&q_0\\
q_1&0&p_1&0&\cdots&0&0&0&0\\
0&q_1&0&p_1&\cdots&0&0&0&0\\
\vdots&\vdots&\vdots&\vdots& &\vdots&\vdots&\vdots&\vdots\\
0&0&0&0&\cdots&q_1&0&p_1&0\\
0&0&0&0&\cdots&0&q_1&0&p_1\\
p_1&0&0&0&\cdots&0&0&q_1&0\end{pmatrix}
\end{equation*}
with $p_0$ and $p_1$ as in \eqref{rho-param} and $q_0:=1-p_0$ and $q_1:=1-p_1$, and the $r\times r$ payoff matrix $\bm W$ is given by
\begin{equation*}\label{W}
\bm W=\begin{pmatrix}0&1&\phantom{_0}0\phantom{_0}&\phantom{_0}0\phantom{_0}&\cdots&0&0&0&-1\\
-1&0&1&0&\cdots&0&0&0&0\\
0&-1&0&1&\cdots&0&0&0&0\\
\vdots&\vdots&\vdots&\vdots& &\vdots&\vdots&\vdots&\vdots\\
0&0&0&0&\cdots&-1&0&1&0\\
0&0&0&0&\cdots&0&-1&0&1\\
1&0&0&0&\cdots&0&0&-1&0\end{pmatrix}.
\end{equation*}

The transition matrix $\bm P$ is irreducible and aperiodic if $r$ is odd, in which case the theorem applies directly.  But if $r$ is even, then $\bm P$ is irreducible and periodic with period 2.  In that case we need the following extension of Theorem~\ref{SLLN}.

\begin{theorem}\label{SLLN2}
Theorem~\ref{SLLN} holds with ``is irreducible and aperiodic'' replaced by ``is irreducible and periodic with period $2$''.
\end{theorem}

We remark that an alternative proof of a strong law of large numbers for Parrondo games could be based on the renewal theorem; see Pyke (2003).

\begin{proof}
The irreducibility and aperiodicity in Theorem~\ref{SLLN} ensures that the Markov chain, with initial distribution equal to the unique stationary distribution, is a stationary strong mixing sequence (Bradley (2005), Theorem 3.1).  Here we must deduce this property in a different way.

The assumption that $\bm P = (P_{ij})_{i,j \in \Sigma}$ is irreducible with period $2$ implies that $\Sigma$ is the disjoint union of $\Sigma_1$ and $\Sigma_2$, and transitions under $\bm P$ take $\Sigma_1$ to $\Sigma_2$ and $\Sigma_2$ to $\Sigma_1$.  This tells us that $\bm P^2$ is reducible with two recurrent classes, $\Sigma_1$ and $\Sigma_2$, and no transient states.  Let the row vectors $\bm\pi_1=(\pi_1(i))_{i \in \Sigma}$ and $\bm\pi_2=(\pi_2(j))_{j \in \Sigma}$ be the unique stationary distributions of $\bm P^2$ concentrated on $\Sigma_1$ and $\Sigma_2$, respectively.  Then $\bm\pi_1\bm P=\bm\pi_2$ and $\bm\pi_2\bm P=\bm\pi_1$, and $\bm\pi:=\frac12(\bm\pi_1+\bm\pi_2)$ is the unique stationary distribution of $\bm P$.

We consider two Markov chains, one in $\Sigma_1\times\Sigma_2$ and the other in $\Sigma_2\times\Sigma_1$, both denoted by $\{(X_0,X_1),(X_2,X_3),(X_4,X_5),\ldots\}$.  The transition probabilities are of the form $P^*((i,j),(k,l)):=P_{jk}P_{kl}$ in both cases.  To ensure that the Markov chains are irreducible, we change the state spaces to $S_1:=\{(i,j)\in\Sigma_1\times\Sigma_2: P_{ij}>0\}$ and $S_2:=\{(j,k)\in\Sigma_2\times\Sigma_1: P_{jk}>0\}$.  The unique stationary distributions are $\bm\pi_1^*$ and $\bm\pi_2^*$ given by
$$
\pi_1^*(i,j)=\pi_1(i)P_{ij}\quad\text{and}\quad\pi_2^*(j,k)=\pi_2(j)P_{jk}.
$$
To check stationarity, we confirm that for each $(k,l)\in S_1$,
\begin{align*}
\sum_{(i,j)\in S_1}\pi_1^*(i,j)P^*((i,j),(k,l))&=\sum_{j\in\Sigma_2}\sum_{i\in\Sigma_1}\pi_1(i)P_{ij}P_{jk}P_{kl}\\
&=\sum_{j\in\Sigma_2}\pi_2(j)P_{jk}P_{kl}=\pi_1(k)P_{kl}=\pi_1^*(k,l).
\end{align*}
An analogous calculation applies to $\bm\pi_2^*$.

We claim that $\bm P^*$ is irreducible and aperiodic on $S_1$ as well as on $S_2$.  It suffices to show that all entries of $(\bm P^*)^n$ are positive on $S_1\times S_1$ and on $S_2\times S_2$ for sufficiently large $n$.  Indeed, given $(i_0,j_0),(i_n,j_n)\in S_1$,
\begin{align*}
&(\bm P^*)^n_{(i_0,j_0) (i_n,j_n)}\\
&=\sum_{(i_1,j_1),(i_2,j_2),\ldots,(i_{n-1},j_{n-1})\in S_1}P^*((i_0,j_0),(i_1,j_1))
P^*((i_1,j_1),(i_2,j_2))\cdots \\
& \hspace{2.5in} \cdots P^*((i_{n-1},j_{n-1}),(i_n,j_n))\\
&=\sum_{(i_1,j_1),(i_2,j_2),\ldots,(i_{n-1},j_{n-1})\in S_1}P_{j_0 i_1}P_{i_1 j_1}P_{j_1 i_2}P_{i_2 j_2}
\cdots P_{j_{n-1} i_n}P_{i_n j_n}\\
&=\sum_{i_1\in\Sigma_1}P_{j_0 i_1}(\bm P)^{2(n-1)}_{i_1 i_n}P_{i_n j_n}>0
\end{align*}
since all entries of $\bm P^{2(n-1)}$ are positive on $\Sigma_1\times\Sigma_1$ for sufficiently large $n$.  A similar argument applies to $S_2$.

Now we compute mean profit at stationarity.  Starting from $\bm\pi_1^*$ we have
\begin{align*}
&E_{\bm\pi_1^*}[w(X_0,X_1)+w(X_1,X_2)]\\
&\qquad{}=\sum_{(i,j)\in S_1}\sum_{(k,l)\in S_1}\pi_1^*(i,j)P^*((i,j),(k,l))[w(i,j)+w(j,k)]\\
&\qquad{}=\sum_{i\in\Sigma_1}\sum_{j\in\Sigma_2}\sum_{k\in\Sigma_1}\pi_1(i)P_{ij}P_{jk}[w(i,j)+w(j,k)]\\
&\qquad{}=\sum_{i,j\in\Sigma}\pi_1(i)P_{ij}w(i,j)+\sum_{j,k\in\Sigma}\pi_2(j)P_{jk}w(j,k)\\
&\qquad{}=\bm\pi_1\dot{\bm P}\bm1+\bm\pi_2\dot{\bm P}\bm1\\
&\qquad{}=2\bm\pi\dot{\bm P}\bm1,
\end{align*}
and the same result holds starting from $\bm\pi_2^*$.

We conclude that, starting with initial distribution $\bm\pi_1^*$, $(X_0,X_1),(X_2,X_3),\break(X_4,X_5),\ldots$ is a stationary strong mixing sequence with a geometric rate, hence the same is true of $w(X_0,X_1)+w(X_1,X_2), w(X_2,X_3)+w(X_3,X_4),\ldots$.

As in Ethier and Lee (2009), the SLLN applies and
$$
\lim_{n\to\infty}(2n)^{-1}S_{2n}=\frac12\,2\bm\pi\dot{\bm P}\bm1=\bm\pi\dot{\bm P}\bm1\text{ a.s.}
$$
The same is true starting with initial distribution $\bm\pi_2^*$, and the coupling argument used by Ethier and Lee (2009) to permit an arbitrary initial state extends to this setting as well.
\end{proof}

\section{Stationary distribution of the random walk on the $n$-cycle}\label{sec-cycle}

We will need to find the stationary distribution of the general random walk on the $n$-cycle ($n$ points arranged in a circle and labeled $0,1,2,\ldots,n-1$) with transition matrix
\begin{equation}\label{P-general}
\tabcolsep=1mm
\bm P:=\begin{pmatrix}0&p_0&\phantom{_{00}}0\phantom{_{00}}&0&\cdots&0&0&0&q_0\\
                      q_1&0&p_1&0&\cdots&0&0&0&0\\
                      0&q_2&0&p_2&\cdots&0&0&0&0\\
                      \vdots&\vdots&\vdots&\vdots& &\vdots&\vdots&\vdots&\vdots\\
                      0&0&0&0&\cdots&q_{n-3}&0&p_{n-3}&0\\
                      0&0&0&0&\cdots&0&q_{n-2}&0&p_{n-2}\\
                      p_{n-1}&0&0&0&\cdots&0&0&q_{n-1}&0\end{pmatrix},
\end{equation}
where $p_i\in(0,1)$ and $q_i:=1-p_i$.  It is possible that a formula has appeared in the literature, but we were unable to find it.  We could derive a more general result with little additional effort by replacing the diagonal of $\bm P$ by $(r_0,r_1,\ldots,r_{n-1})$, where $p_i>0$, $q_i>0$, $r_i\ge0$, and $p_i+q_i+r_i=1$ ($i=0,1,\ldots,n-1$).  But to minimize complications, we treat only the case of \eqref{P-general}.

The transition matrix $\bm P$ is irreducible and its unique stationary distribution $\bm\pi=(\pi_0,\pi_1,\ldots,\pi_{n-1})$ satisfies $\bm\pi=\bm\pi\bm P$ or
$$
\pi_i=\pi_{i-1}p_{i-1}+\pi_{i+1}q_{i+1},\quad i=1,2,\ldots,n-1,
$$
where $\pi_n:=\pi_0$ and $q_n:=q_0$, 
or 
\begin{equation*}
\pi_{i-1}p_{i-1}-\pi_i q_i=\pi_i p_i-\pi_{i+1}q_{i+1},\quad i=1,2,\ldots,n-1.
\end{equation*}
Thus, $\pi_{i-1} p_{i-1}-\pi_i q_i=C$, a constant, for $i=1,2,\ldots,n$, where $\pi_n:=\pi_0$ and $q_n:=q_0$;  alternatively,
\begin{equation}\label{statdist-recursion}
\pi_i=-\frac{C}{q_i}+\frac{p_{i-1}}{q_i}\,\pi_{i-1}.
\end{equation}

This is of the form $x_i=a_i+b_i x_{i-1}$, $i=1,2,\ldots$, the solution of which is
$$
x_i=\sum_{j=1}^i a_j\bigg(\prod_{k=j+1}^i b_k\bigg)+\bigg(\prod_{j=1}^i b_j\bigg)x_0, \quad i=1,2,\ldots,
$$
where empty products are 1.  Applying this to \eqref{statdist-recursion}, we find that
\begin{align*}
\pi_i &= -C\sum_{j=1}^i \frac{1}{q_j} \bigg(\prod_{k=j+1}^i \frac{p_{k-1}}{q_k}\bigg) + \bigg(\prod_{j=1}^i \frac{p_{j-1}}{q_j}\bigg) \pi_0 \\
&= -C \,\frac{1}{q_i}\bigg[1+\sum_{j=1}^{i-1} \bigg(\prod_{k=j}^{i-1} \frac{p_k}{q_k}\bigg)\bigg] + \frac{q_0}{q_i}\bigg(\prod_{j=0}^{i-1} \frac{p_j}{q_j}\bigg)  \pi_0,\quad i=1,2,3,\ldots,n.
\end{align*}
In particular, $C$ can be determined in terms of $\pi_0$ from the $i=n$ case (since $\pi_n:=\pi_0$ and $q_n:=q_0$).  It is given by
\begin{equation*}
C=q_0\bigg[\bigg(\prod_{j=0}^{n-1} \frac{p_j}{q_j}\bigg)-1\bigg]\bigg[1+\sum_{j=1}^{n-1} \bigg(\prod_{k=j}^{n-1} \frac{p_k}{q_k}\bigg) \bigg]^{-1}\pi_0.
\end{equation*}
Defining $\Pi_0:=1$ and 
\begin{align}\label{Pi_i}
\Pi_i&:=\!-\frac{q_0}{q_i}\bigg[\bigg(\prod_{j=0}^{n-1}\frac{p_j}{q_j}\bigg)-1\bigg]\bigg[1+\sum_{j=1}^{n-1}\bigg(\prod_{k=j}^{n-1}\frac{p_k}{q_k}\bigg)\bigg]^{-1}\bigg[1+\sum_{j=1}^{i-1}\bigg(\prod_{k=j}^{i-1}\frac{p_k}{q_k}\bigg)\bigg]\nonumber
\\
&\hspace{7cm}{}+\frac{q_0}{q_i}\bigg(\prod_{j=0}^{i-1}\frac{p_j}{q_j}\bigg)
\end{align}
for $i=1,2,\ldots,n-1$, we find that $\pi_i=\Pi_i\pi_0$ for $i=0,1,\ldots,n-1$, and the following lemma is immediate.

\begin{lemma}
The unique stationary distribution $\bm\pi=(\pi_0,\pi_1,\ldots,\pi_{n-1})$ of the transition matrix $\bm P$ of \eqref{P-general} is given by
$$
\pi_i=\frac{\Pi_i}{\Pi_0+\Pi_1+\cdots+\Pi_{n-1}},\quad i=0,1,\ldots,n-1,
$$
where $\Pi_0:=1$ and $\Pi_i$ is defined by \eqref{Pi_i} for $i=1,2,\ldots,n-1$.
\end{lemma}

\begin{example}\label{p,p,p}
As a check of the formula, consider the case in which $p_0=p_1=\cdots=p_{n-1}=p\in(0,1)$ and $q_0=q_1=\cdots=q_{n-1}=q:=1-p$.  Here the transition matrix is doubly stochastic, so the unique stationary distribution is discrete uniform on $\{0,1,\ldots,n-1\}$.  Indeed, algebraic simplification shows that $\Pi_0=\Pi_1=\cdots=\Pi_{n-1}=1$.
\end{example}

\begin{example}\label{p0,p,p}
Consider next the case in which $p_1=p_2=\cdots=p_{n-1}=p\in(0,1)$ and $q_1=q_2=\cdots=q_{n-1}=q:=1-p$.  Of course $p_0$ and $q_0:=1-p_0$ may differ from $p$ and $q$.  Then $\Pi_0:=1$ and
\begin{equation}\label{Pi_i,Ex2}
\Pi_i=-\bigg[\frac{(p_0/q)(p/q)^{n-1}-q_0/q}{(p/q)^n-1}\bigg]((p/q)^i-1)+(p_0/q)(p/q)^{i-1}
\end{equation}
for $i=1,2,\ldots,n-1$.  It follows that
\begin{align}\label{Pi-sum,Ex2}
\sum_{i=0}^{n-1}\Pi_i&=1-\bigg[\frac{(p_0/q)(p/q)^{n-1}-q_0/q}{(p/q)^n-1}\bigg]\bigg[\frac{(p/q)((p/q)^{n-1}-1)}{p/q-1}-(n-1)\bigg]\nonumber\\
&\qquad{}+\frac{(p_0/q)((p/q)^{n-1}-1)}{p/q-1}\nonumber\\
&=1-\frac{p_0-q_0}{p-q}+n\,\frac{p_0p^{n-1}-q_0q^{n-1}}{p^n-q^n},
\end{align}
where the last step involves some algebra and we have implicitly assumed that $p\ne\frac12$.  In particular, $\pi_0$ is the reciprocal of \eqref{Pi-sum,Ex2}.  This result is useful in evaluating $\mu(r,\rho,\gamma A+(1-\gamma)B)$; see Section~\ref{sec-rate}.
\end{example}

\begin{example}\label{1-p,p,p}
Consider finally the special case of Example \ref{p0,p,p} in which $p_0=q$ and $q_0=p$.  The \eqref{Pi_i,Ex2} becomes
\begin{equation*}\label{Pi_i,Ex3}
\Pi_i=-\bigg[\frac{(p/q)((p/q)^{n-2}-1)}{(p/q)^n-1}\bigg]((p/q)^i-1)+(p/q)^{i-1}
\end{equation*}
for $i=1,2,\ldots,n-1$, and \eqref{Pi-sum,Ex2} becomes
\begin{align}\label{Pi-sum,Ex3}
\sum_{i=0}^{n-1}\Pi_i =2+npq\,\frac{p^{n-2}-q^{n-2}}{p^n-q^n}.
\end{align}
We have again implicitly assumed that $p\ne\frac12$, and again $\pi_0$ is the reciprocal of \eqref{Pi-sum,Ex3}.  This result is useful in evaluating $\mu(r,0,\gamma A+(1-\gamma)B)$; see Section~\ref{sec-rate}.
\end{example}

\section{Evaluation of rate of profit}\label{sec-rate}

Recall that mean profit has the form $\mu=\bm\pi\dot{\bm P}\bm1$, which we apply to $\bm P:=\gamma\bm P_A+(1-\gamma)\bm P_B$.  

To find $\mu(r,\rho,\gamma A+(1-\gamma)B)$, it suffices to note that $\bm P$ has the form \eqref{P-general} under the assumptions of Example~\ref{p0,p,p} with $n:=r$,
\begin{equation}\label{p,p0}
p:=\frac{\gamma}{2}+(1-\gamma)\frac{1}{1+\rho},\quad\text{and}\quad p_0:=\frac{\gamma}{2}+(1-\gamma)\frac{\rho^{r-1}}{1+\rho^{r-1}},
\end{equation}
where $0<\rho<1$.  Thus,
\begin{equation}\label{mu(r,rho)}
\mu(r,\rho,\gamma A+(1-\gamma)B)=\pi_0(p_0-q_0)+(1-\pi_0)(p-q),
\end{equation}
with $\pi_0$ being the reciprocal of \eqref{Pi-sum,Ex2}.

To find $\mu(r,0,\gamma A+(1-\gamma)B)$, it suffices to note that $\bm P$ has the form \eqref{P-general} under the assumptions of Example~\ref{p0,p,p} with $n:=r$,
$$
p:=\frac{\gamma}{2}+(1-\gamma)1=1-\frac{\gamma}{2},\quad\text{and}\quad p_0:=\frac{\gamma}{2}+(1-\gamma)0=\frac{\gamma}{2}=1-p=q.
$$
We are therefore in the setting of Example~\ref{1-p,p,p}, and
\begin{equation}\label{mu(r,0)}
\mu(r,0,\gamma A+(1-\gamma)B)=\pi_0(q-p)+(1-\pi_0)(p-q)=(p-q)(1-2\pi_0),
\end{equation}
with $\pi_0$ being the reciprocal of \eqref{Pi-sum,Ex3}.

\begin{proof}[Proof of Theorem~\ref{rate-random}]
From \eqref{mu(r,0)} and \eqref{Pi-sum,Ex3} with $n=r$, we have
\begin{align*}
\mu(r,0,\gamma A+(1-\gamma)B)&=(p-q)(1-2\pi_0)\\
&=(p-q)\bigg(1-\frac{2(p^r-q^r)}{2(p^r-q^r)+rpq(p^{r-2}-q^{r-2})}\bigg)\\
&=\frac{rpq(p-q)(p^{r-2}-q^{r-2})}{2(p^r-q^r)+rpq(p^{r-2}-q^{r-2})},
\end{align*}
and the theorem follows by substituting $1-\gamma/2$ and $\gamma/2$ for $p$ and $q$.
\end{proof}

\begin{proof}[Proof of Corollary~\ref{rate-random-asymp}]
We want to show that $\mu(r,0,\gamma A+(1-\gamma)B)$ can be close to 1 by choosing $p:=1-\gamma/2$ close to 1 and $\pi_0$ close to 0, which requires $r$ large.  So we consider a sequence $p\to 1$ as $r\to\infty$.  In this case,
\begin{align*}
\pi_0
&=\frac{p^r-q^r}{2(p^r-q^r) + rpq(p^{r-2}-q^{r-2})}\\
&\sim\frac{p^r}{2 p^r +r qp^{r-1}}\\
&=\frac{p}{2p+rq}.
\end{align*}
Now let us specify that $p=1-1/\sqrt{r}$ (equivalently, $\gamma=2/\sqrt{r}$).  Then, by \eqref{mu(r,0)},
\begin{align*}
1-\mu(r,0,\gamma A+(1-\gamma)B)&=1-(p-q)(1-2\pi_0)\\
&\sim1-(1-2/\sqrt{r})\bigg(1-\frac{2(1-1/\sqrt{r})}{2(1-1/\sqrt{r})+\sqrt{r}}\bigg)\\
&\sim\frac{4}{\sqrt{r}},
\end{align*}
as required.
\end{proof}

\begin{proof}[Proof of Corollary~\ref{rate-periodic-asymp}]
For even $r\ge4$ and positive integers $s\le r/2$, Theorem~\ref{rate-periodic} implies that
\begin{align*}
&1-\mu(r,0,(AB)^s B^{r-2})\\
&\qquad{}=1-\bigg(1-\frac{2(s-1)}{r+2(s-1)}\bigg)\bigg(1-\frac{1}{2^s}\bigg)\\
&\qquad{}=\frac{2s}{r+2(s-1)}-\frac{2}{r+2(s-1)}+\frac{1}{2^s}-\frac{2(s-1)}{r+2(s-1)}\cdot\frac{1}{2^s},
\end{align*}
if initial capital is even. With $s$ replaced by $s_r:=\lfloor\log_2 r\rfloor-1$, the first term is asymptotic to $2s_r/r$ as $r\to\infty$ and the remaining terms are $O(1/r)$.

For odd $r\ge3$, the argument is essentially the same.
\end{proof}

\begin{proof}[Proof of Theorem~\ref{sup=1,random}]
It is enough to show that $\mu(r,\rho,\gamma A+(1-\gamma)B)$ is continuous at $\rho=0$ for fixed $r$ and $\gamma$.  In fact, there is a complicated but explicit formula, given by \eqref{mu(r,rho)}, using \eqref{Pi-sum,Ex2} and \eqref{p,p0}, showing that it is a rational function of $\rho$.  Therefore, we need only show that it does not have a pole at $\rho=0$.  In fact, Theorem~\ref{rate-random} shows that $\mu(r,0,\gamma A+(1-\gamma)B)$ is the ratio of two positive numbers, and this is sufficient. 
\end{proof}

\section*{Acknowledgments}

We are grateful to Derek Abbott for raising the question addressed here.  SNE was partially supported by a grant from the Simons Foundation (429675).  JL was supported by the Basic Science Research Program through the National Research Foundation of Korea (NRF) funded by the Ministry of Education (NRF-2018R1D1A1B07042307).

\begin{newreferences}

\item
Ajdari, A. and Prost, J. (1992). Drift induced by a spatially periodic potential of low symmetry: Pulsed dielectrophoresis. \textit{C. R. Acad.\ Sci., S\'erie 2} \textbf{315}, 1635--1639.

\item
Bradley, R. C. (2005). Basic properties of strong mixing conditions. A survey and some open questions. \textit{Probab.\ Surveys} \textbf{2} 107--144.

\item
Dinis, L. (2008). Optimal sequence for Parrondo games.  \textit{Phys.\ Rev.\ E}  \textbf{77}, 021124. 

\item
Ethier, S. N. and Lee, J. (2009). Limit theorems for Parrondo's paradox.  \textit{Electronic J. Probab.} \textbf{14} (62), 1827--1862. 

\item
Ethier, S. N. and Lee, J. (2019). How strong can the Parrondo effect be?  \textit{J. Appl.\ Probab.} \textbf{56} (4), 1198--1216.

\item
Harmer, G. P. and Abbott, D. (1999). Parrondo's paradox. \textit{Statist.\ Sci.} \textbf{14} (2), 206--213.

\item
Pyke, R. (2003). On random walks and diffusions related to Parrondo's games. In \textit{Mathematical Statistics and Applications: Festschrift for Constance Van Eeden}, ed. M. Moore, S. Froda, and C. L\'eger.  IMS Lecture Notes--Monograph Series \textbf{42}, Institute of Mathematical Statistics, Beachwood, OH, pp.~185--216. 

\end{newreferences}

\end{document}